\documentclass[12pt]{amsart}
\usepackage{amssymb,latexsym}
\usepackage{amsrefs}

\theoremstyle{plain}
\newtheorem*{thm}{Theorem}
\newtheorem*{lem}{Lemma}

\newcommand{\R}{\mathbb R}

\newcommand{\rbar}{\overline\R}
\newcommand{\udim}{\overline{\dim}_{\mathrm{M}}}

\makeatletter
   \def\th@plain{\slshape}
   \makeatother

\begin{document}
\title[A dichotomy for  expansions of the real field]
{A dichotomy for  expansions of the real field}

\author
{Antongiulio Fornasiero}
\address
{Institut f\"ur Mathematische Logik\\
 Einsteinstr. 62\\
 48149 M\"unster, Germany}
\email{antongiulio.fornasiero@googlemail.com}
\urladdr{http://www.dm.unipi.it/\textasciitilde fornasiero}

\author
{Philipp Hieronymi}
\address
{Department of Mathematics\\University of Illinois at Urbana-Champaign\\1409 West Green Street\\Urbana, IL 61801}
\email{p@hieronymi.de}
\urladdr{http://www.math.uiuc.edu/\textasciitilde phierony}

\author
{Chris Miller}
\address{Department of Mathematics\\
The Ohio State University\\
231 West 18th Avenue\\
Columbus, OH 43210} \email{miller@math.osu.edu}
\urladdr{http://www.math.osu.edu/\textasciitilde miller}

 \begin{abstract}
A dichotomy for expansions of the real field is established: Either $\mathbb Z$ is definable or
every nonempty bounded nowhere dense definable subset of $\mathbb R$ has Minkowski dimension zero.
\end{abstract}

\thanks{\today. Some version of this document has been accepted for publication in the Proceedings of the American Mathematical Society, but has not yet gone to press. Comments are welcome. Miller partly supported by NSF DMS-1001176.}




\maketitle

Given $E\subseteq \mathbb R^n$ bounded and $r>0$, let $N(E,r)$ be the number of closed balls of radius $r$ needed to cover $E$.
Put
$
\udim E=\varlimsup_{r\downarrow 0}\log N(E,r)/\log (1/r)
$
(with $\log 0:=-\infty$),
the \textbf{upper Minkowski dimension} of $E$ (but there are many different names
and equivalent formulations in the literature).
We refer to Falconer~\cite{falconerbook} for basic facts.
We say that $E$ is \textbf{M-null} if $\udim E\leq 0$.

\begin{thm}
Given an expansion of the real field $\rbar:=(\mathbb R,+,\cdot\,)$\textup, either $\mathbb Z$ is definable or every bounded nowhere dense definable subset of $\mathbb R$ is M-null.
\end{thm}

In other words:
\textsl{An expansion of $\rbar$ avoids defining $\mathbb Z$
if and only if every bounded definable subset of $\mathbb R$ is either somewhere dense or M-null.}
(See van den Dries and Miller~\cite{geocat}*{\S2} for definitions and basic facts about definability in expansions of $\rbar$.)
A standard exercise is that $\udim\{\,1/(n+1): n\in \mathbb N\,\}=1/2$, so only the forward implication needs to established.
We begin with a result in geometric measure theory.
For $E\subseteq \mathbb R$, put
$$
Q(E)=\{\, (a-b)/(c-d): a,b,c,d\in E \And c\neq d\,\}.
$$

\begin{lem}
Let $E\subseteq \mathbb R$ be bounded and $\udim E>0$. Then there exist $n\in \mathbb N$ and linear $T\colon \mathbb R^n\to \mathbb R$ such that  $Q(T(E^n))$ is dense in $\mathbb R$.
\end{lem}

\begin{proof}
That $\udim E^n=n\,\udim E$ is an exercise, so  $\lim_{n\to \infty}\udim E^n=\infty$.
By Falconer and Howroyd~\cite{proj}*{Theorem~3}, there exist $n\in \mathbb N$ and linear
$T\colon \mathbb R^n\to\mathbb R$ such that $\udim T(E^n)>1/2$.
Hence, it suffices to consider the case that $\udim E>1/2$ and
show that $Q(E)$ is dense in $\mathbb R$.
(We thank Kenneth Falconer for the following elegant geometric argument.
Our original proof was based on additive combinatorics.)
Suppose not.
Observe that $Q(E)$ is the set of slopes of nonvertical lines connecting pairs of points in $E^2$.
Thus, the difference set $\{\, u-v: u,v\in E^2\,\}$ of $E^2$ is disjoint
from some open double cone $C\subseteq \mathbb R^2$ centered at the origin.
Let $\ell$ be the line through the origin perpendicular to the axis of~$C$.
Then the restriction to $E^2$ of the projection of $\mathbb R^2$ onto $\ell$ is injective, and the compositional inverse is Lipshitz.
Hence, $E^2$ is contained in a rotation of the graph of a Lipshitz function from some bounded subinterval of $\mathbb R$ into $\mathbb R$.
It follows from~\cite{falconerbook}*{Corollary~11.2} that $\dim E^2\leq 1$.
But then $\udim E=(\udim E^2)/2\leq 1/2$, a contradiction.
\end{proof}

\begin{proof}[Proof of Theorem]
It suffices to let $E\subseteq \mathbb R$ be bounded and nowhere dense with $\udim E>0$ and show that $(\rbar,E)$ defines $\mathbb Z$.
As $\udim$ is preserved under topological closure, we reduce to the case that $E$ is compact and has no interior.
The set $A$ of left endpoints of the complementary intervals of $E$ is dense in $E$, so $\udim A>0$.
By the lemma, there exist $n\in \mathbb N$ and linear $T\colon \mathbb R^n\to\mathbb R$ such that $Q(T(A^n))$ is dense in $\mathbb R$.
Let
$D$ be the set of midpoints of the bounded complementary intervals of $E$.
Define $g\colon D\to \mathbb R$
by $g(x)=\sup(E\cap (-\infty,x])$.
As $g$ is definable in $(\rbar,E)$ and $g(D)=A$, there exist $m\in \mathbb N$ and a function $f\colon D^m\to\mathbb R$ definable in $(\rbar,E)$ such that $f(D^m)$ is somewhere dense.
As $D$ is discrete,  $(\rbar,E)$ defines $\mathbb Z$ by~\cite{hier2}*{Theorem~A}.
\end{proof}

We believe the result to be optimal.
There are Cantor sets $K\subseteq \mathbb R$ such that $(\rbar,K)$ defines sets in every projective level, yet every subset of $\mathbb R$ definable in $(\rbar,K)$ either has interior or is nowhere dense (Friedman \textit{et al.}~\cite{fkms}); we suspect that, for at least some such $K$, there exist dense $X\subseteq \mathbb R$ having empty interior such that $(\rbar,K,X)$ still does not define $\mathbb Z$.
In any case, $\mathbb Z$ is not definable in the expansion of $\rbar$ by $\{\,(2^j,2^k3^l):j,k,l\in\mathbb Z\,\}$ (G\"{u}nayd\i n~\cite{G}), yet it evidently defines both an infinite discrete set and a dense subset of $\mathbb R^{>0}$ that has empty interior.

Consequences, extensions and variations of the theorem are numerous, but will be detailed elsewhere by various subsets of the authors.

\bibsection

\end{document}